\documentclass[11pt]{amsart}
\pdfoutput=1

\usepackage[colorlinks=true, citecolor=black, filecolor=black, linkcolor=black, urlcolor=black]{hyperref}

\newcommand\subsubsec[1]{\medskip\noindent\textbullet~\emph{#1.}~}
\newcommand\arXiv[1]{\href{http://arxiv.org/abs/#1}{arXiv:#1}}

\def\Gchiral{G^+}
\def\Phiplus{\Phi^+}

\def\arc{\stackrel{\textstyle\frown}}
\def\bp{{\bar\partial}}
\def\pa{\partial}
\def\sm{\setminus}

\def\wh{\widehat}
\def\wt{\widetilde}
\def\ve{\varepsilon}

\def\hol{ \mathrm{hol}}
\def\Aut{ \mathrm{Aut}}
\def\id{ \mathrm{id}}
\def\SLE{ \mathrm{SLE}}

\def\Im{ \mathrm{Im}}

\DeclareMathOperator{\Sing}{Sing}

\def\FF{\mathcal{F}}
\def\LL{\mathcal{L}}
\def\VV{\mathcal{V}}
\def\XX{\mathcal{X}}
\def\YY{\mathcal{Y}}

\def\C{\mathbb{C}}

\def\E{\mathbf{E}}
\def\H{\mathbb{H}}
\def\P{\mathbf{P}}
\def\R{\mathbb{R}}
\def\S{\mathbb{S}}

\theoremstyle{plain}
\newtheorem*{thm*}{Theorem}
\newtheorem{thm}{Theorem}[section]
\newtheorem{lem}[thm]{Lemma}

\newtheorem{prop}[thm]{Proposition}

\theoremstyle{definition}
\newtheorem*{eg*}{Example}
\newtheorem*{egs*}{Examples}
\newtheorem*{def*}{Definition}

\theoremstyle{remark}
\newtheorem*{rmk*}{Remark}
\newtheorem*{rmks*}{Remarks}

\numberwithin{equation}{section}

\begin{document}
\title[CFT of dipolar SLE(4) with mixed boundary condition]{Conformal field theory of dipolar SLE(4) with mixed boundary condition}

\author{Nam-Gyu Kang}
\address{Department of Mathematical Sciences, 
Seoul National University, 
Seoul 151-747, Republic of Korea}
\email{nkang@snu.ac.kr}
\thanks{The author was partially supported by NRF grant 2010-0021628. 
The author also holds joint appointment in the Research Institute of Mathematics, Seoul National University. \vspace{.5in}}

\date{}


\begin{abstract}
We develop a version of dipolar conformal field theory in a simply connected domain with the Dirichlet-Neumann boundary condition and central charge one.
We prove that all correlation functions of the fields in the OPE family of Gaussian free field with a certain boundary value are martingale-observables for dipolar SLE(4).
\end{abstract}

\maketitle

\section{Introduction}
In this paper we consider a version of dipolar conformal field theory with central charge one $(c=1)$ in a simply connected domain $D$ with two marked boundary points $q=q_-, r=q_+.$
More precisely, the theory we develop is based on a certain family of fields generated by the Gaussian free field in $D$ with the Dirichlet boundary condition on $\arc{qr}$ and the Neumann boundary condition on the other boundary arc.
We prove that the correlation functions of fields in this family under the insertion of a certain boundary condition changing operator form a collection of martingale-observables for dipolar SLE(4).

We apply definitions and theories developed in \cite{KM11} and \cite{KM12} to a different
conformal setting with a different boundary condition.
In the chordal case (\cite{KM11}) and the radial case (\cite{KM12}), we consider a simply connected domain $D$ with a marked boundary point and a marked interior point, respectively.
Both theories are based on the Gaussian free field with the Dirichlet boundary condition.
However, their central charge modifications are different.
In this paper we do not discuss the central charge modification of dipolar conformal field theory with mixed boundary condition.

We also explain the differences between the conformal field theory we consider in this paper and the other theories in \cite{KM11} and \cite{KM12}.
Unlike the chordal and the radial cases, the current field (the derivative of Gaussian free field) is real on the boundary arc with the Neumann boundary condition.
(As in the chordal and the radial cases, it is purely imaginary on the boundary arc with the Dirichlet boundary condition.)
However, the Virasoro field is real on the whole boundary and Ward's equations have the similar forms as in the chordal and the radial cases.
On the other hand, the Virasoro field has a double pole at each of the marked boundary points.
In the radial case (\cite{KM12}), it has a double pole at a marked interior point for $c<1$ and  has no singularity at this point for $c=1.$
In the forthcoming paper (\cite{KT13}), we study a dipolar conformal field theory $(c\le1)$ with the Dirichlet boundary condition and relate it to dipolar $\SLE(\kappa).$

\medskip \section{Main result}

\subsection{Definition of dipolar SLE} For a simply connected domain $(D,p,q_-,q_+)$ with three marked boundary points $p,q=q_-,r=q_+$ (a point $p$ is in the positively oriented boundary arc $\arc{qr}$ from $q$ to $r$), dipolar Schramm-Loewner evolution (SLE) in $(D,p,q_-,q_+)$ with a positive parameter $\kappa$ is the conformally invariant law on random curves from $p$ to $\arc{rq}$ described by the solution $g_t(z)$ (which exists up to a stopping time $\tau_z\in(0,\infty]$) of the dipolar Loewner equation
\begin{equation} \label{eq: g}
\partial_t g_t(z) = \frac1{\tanh((g_t(z)-\xi_t)/2)}, \quad (\xi_t = \sqrt\kappa B_t).
\end{equation}
Here, $B_t$ is a standard Brownian motion with $B_0=0$ and $g_0$ is a given conformal map from $D$ onto the strip $\S:=\{z = x+ iy\,|\, 0 < y < \pi\}$ such that $g_0(p) = 0,$ $ g_0(q_\pm)=\pm\infty.$
Then for all $t,$
$$w_t(z):=g_t(z)-\xi_t$$
is a well-defined conformal map from
$$D_t := \{z \in D: \tau_z>t\}$$
onto the strip $\S.$
The \emph{dipolar SLE curve} $\gamma$ is defined by the equation
$$\gamma_t \equiv \gamma(t) := \lim_{z\to0} w_t^{-1}(z).$$
Then $w_t$ is a random conformal map from $(D_t,\gamma_t,q_-,q_+)$ onto $(\S,0,-\infty,\infty).$
See \cite{Zhan04} for the basic properties of dipolar SLE.

\medskip \subsection{Definition of martingale-observables and OPE family}
We say that a non-random conformal field $M$ of $n$ variables in $\S$ is a \emph{martingale-observable} for dipolar $\SLE(\kappa)$ if for any $z_1,\cdots ,z_n\in D,$ the process
$$M_t(z_1,\cdots, z_n)=(M_{D_t,\gamma_t,q_-,q_+}\,\|\,\id)(z_1,\cdots, z_n) = (M\,\|\,w_t^{-1})(z_1,\cdots, z_n)$$
(stopped when any $z_j$ exits $D_t$) is a local martingale on SLE probability space.
(Note that $\Aut(D,p,q_-,q_+)$ is a trivial group.)

We denote by $\FF$ the \emph{OPE family} of the Gaussian free field $\Phi$ with mixed boundary condition, which is the algebra (over $\C$) spanned by $1,\pa^j\bp^k\Phi$ and $\pa^j\bp^k e^{*\alpha \Phi}\,(\alpha\in\C)$ under the OPE multiplication.
See \cite[Section~3.2]{KM11} for the definition of OPE multiplication and Subsection~\ref{ss: V and T} below for the definition of $e^{*\alpha \Phi}.$

\medskip \subsection{Statement of main theorem}
We prove that the correlation functions of fields in $\FF$ under the insertion of Wick's part of a certain boundary vertex field are martingale-observables for dipolar SLE(4).
A version of the following theorem approached from physics perspective appeared in \cite{BBH05}.

\begin{thm} \label{main}
For any fields $X_j$ in $\FF,$ the non-random field
\begin{equation} \label{eq: main}
\E\,e^{\odot ia \Phi^+(p,q_-)}\, X_1(z_1)\cdots X_n(z_n),\qquad (a = \frac1{\sqrt2})
\end{equation}
is a martingale-observable for dipolar $\SLE(4).$
\end{thm}
Here, $\Phiplus$ is the chiral bosonic field, the ``holomorphic part" of $\Phi.$
See Subsection~\ref{Vplus} for its definition.
Also see \cite{BB03}, \cite{RBGW07}, \cite{KM11} for the chordal version and \cite{BB04}, \cite{Cardy04}, \cite{KM12} for the radial version of this theorem with general $\kappa>0.$

\medskip \section{\texorpdfstring{A dipolar CFT with mixed boundary condition and \textit{c}\,=\,1}{A dipolar CFT}}

Given a simply connected domain $(D,q_-,q_+)$ with two marked boundary points $q_-$ and $q_+,$ we consider the Gaussian free field $\Phi$ whose 2-point function is the Green's function with mixed boundary condition.
As Fock space fields, the non-chiral vertex fields $\VV^\alpha$ and the Virasoro field $T$ are expressed in terms of OPE exponentials of $\Phi$ and the OPE square of the current $J = \pa \Phi,$ respectively.
For dipolar Loewner vector fields $v_\zeta,$ we construct Ward's functionals $W(v_\zeta)$ such that the action of Lie derivative operators $\LL_{v_\zeta}$ on any string $X$ of fields in the Ward family $\FF(T)$ of the Virasoro field $T$ is equivalent to $W(v_\zeta)\,X$ within correlations.
In the last subsection we define the rooted vertex fields $V_\star^\alpha$ and study their basic properties including Ward's equation and the level two degeneracy equation.
We assume that the readers are familiar with definitions of some concepts developed in  \cite[Lectures~1--4]{KM11}.

\medskip \subsection{Gaussian free field} \label{ss: GFF}
We consider a simply connected domain $(D,q_-,q_+)$ with two marked boundary points $q_-=q$ and $q_+=r.$
The Gaussian free field $\Phi$ in $D$ with mixed boundary condition (the Dirichlet boundary condition on the positively oriented boundary arc $\arc{qr}$ and the Neumann boundary condition on the other boundary arc $\arc{rq}$) is a Fock space field such that its 1-point function is trivial, i.e., $\E[\Phi(z)]=0,$ and its 2-point correlation function is given by 
$$\E[\Phi(\zeta)\Phi(z)] = 2G(\zeta,z),$$
where $G$ is the Green's function for $D$ with mixed boundary condition.
(An alternative and more traditional notation for the correlation function is $\langle\cdot\rangle.$)
The boundary arcs $\arc{qr}$ and $\arc{rq}$ are called the Dirichlet boundary and the Neumann boundary, respectively.
In the first quadrant,
$$G(\zeta,z) = \log\Big|\frac{(\zeta-\bar z)(\zeta+z)}{(\zeta-z)(\zeta+\bar z)}\Big|$$
with the Dirichlet boundary condition on the positive real axis $\R_+$ and the Neumann boundary condition on $i\R_+.$
Since the Green's function is conformally invariant,
$$G(\zeta,z)=\log \Big|\tanh\frac{\zeta-\bar z}4\Big|-\log \Big|\tanh\frac{\zeta- z}4\Big|$$
in the strip $\S:=\{z = x+ iy\,|\, 0 < y < \pi\}$
with the Dirichlet boundary condition on the real axis $\R$ and the Neumann boundary condition on the horizontal line $\R + \pi i.$
The current $J :=\pa \Phi$ is also a Fock space field with
\begin{equation} \label{eq: EJPhi}
\E[J(\zeta)\Phi(z)] =
-\frac1{2\sinh\dfrac{\zeta-z}2}+\frac1{2\sinh\dfrac{\zeta-\bar z}2}\qquad\textrm{in }(\S,-\infty,\infty).
\end{equation}
This shows that the current $J$ is purely imaginary on the Dirichlet boundary $\R$ and it is real on the Neumann boundary $\R+\pi i.$
Thus $J \odot J$ is real on the whole boundary.
The current $J$ has the 2-point function
\begin{equation} \label{eq: EJJ}
\E[J(\zeta)J(z)] =
-\frac1{4\sinh\dfrac{\zeta-z}2\tanh\dfrac{\zeta-z}2}\qquad\textrm{in }(\S,-\infty,\infty).
\end{equation}

\medskip \subsection{Vertex fields and Virasoro field} \label{ss: V and T} \hfill

\subsubsec{OPE powers of Gaussian free field} 
Let $\Phi$ be the Gaussian free field in $(D,q_-,q_+)$ with mixed boundary condition described in the previous subsection.
As $\zeta\to z,\;(\zeta\ne z),$ the operator product expansion (OPE) of $\Phi$ and $\Phi$ itself is stated as 
$$\Phi(\zeta) \Phi(z) = \log \frac{1}{|\zeta-z|^2} + 2c(z) + \Phi^{\odot 2}(z) + o(1),$$
where the non-random function $c$ is given by 
$$c=\log\Big(\frac4{|w'|}\tan\frac{\Im\,w}2\Big)$$
and $w$ is a conformal map from $(D,q_-,q_+)$ onto $(\S,-\infty,\infty).$
We call $\Phi^{\odot2} + 2c$ the OPE square of $\Phi$ and write $\Phi^{*2}$ for it.
One can define the OPE powers $\Phi^{*n}$ and then express them in terms of Wick's powers $\Phi^{\odot n}$ of $\Phi:$
$$\Phi^{\odot n} = (2c)^{n/2}H_n^*\left(\frac{\Phi}{\sqrt{2c}}\right),$$
where $H_n(z) = \sum_{k=0}^n a_k z^k$ are the Hermite polynomials (the monic polynomials orthogonal with respect to the standard Gaussian measure $e^{-x^2/2}/\sqrt{2\pi}\,dx$ on $\R$) and
$H_n^*(\alpha\Phi) = \sum_{k=0}^n a_k \alpha^k\Phi^{*k}$
(cf. \cite[Proposition~3.2]{KM11}).

\subsubsec{Vertex fields} 
The (non-chiral) \emph{vertex fields} $\VV^\alpha\,(\alpha\in\C)$ are defined as OPE exponentials of $\Phi:$
$$\VV^\alpha = e^{*\alpha\Phi} = \sum_{n=0}^\infty \alpha^n \frac{\Phi^{*n}}{n!}.$$
In terms of Wick's exponentials of $\Phi,$ we express them as
$$e^{*\alpha\Phi}=C^{\alpha^2}\,e^{\odot\alpha\Phi},\qquad C = e^c =\frac4{|w'|}\tan\frac{\Im\,w}2$$
(cf. \cite[Proposition~3.3]{KM11}).
Thus they are $(-\alpha^2/2,-\alpha^2/2)$-differentials.

\subsubsec{Transformation laws} 
Recall some transformation laws for conformal Fock space fields. 
A conformal Fock space field $F$ is a $(\lambda,\lambda_*)$-\emph{differential} (a pair $(\lambda,\lambda_*)$ is called conformal dimensions or degrees) if for any two overlapping charts $\phi,\wt\phi,$ and for any correlation functional $\YY,$ we have
$$\E\,[\YY(F\,\|\,\phi)] = (h')^\lambda(\overline{h'})^{\lambda_*}\E\,[\YY(F\,\|\,\wt\phi)\circ h],$$
where $h=\wt\phi\circ\phi^{-1}:~ \phi(U\cap\wt U)\to \wt\phi(U\cap\wt U)$ is the transition map. 
For examples, $\Phi$ is a (0,0)-differential,
$J$ is a (1,0)-differential, and $J\odot J$ is a (2,0)-differential. 
On the other hand, \emph{pre-pre-Schwarzian forms}, \emph{pre-Schwarzian forms}, and \emph{Schwarzian forms} of order $\mu(\in\C)$ are fields with transformation laws 
$$\E\,[\YY(F\,\|\,\phi)]=\E\,[\YY(F\,\|\,\wt\phi)\circ h ]+\mu \log h',\, \E\,[\YY(F\,\|\,\phi)]=h'\,\E\,[\YY(F\,\|\,\wt\phi)\circ h] +\mu \frac{h''}{h'},$$ 
and $$\E\,[\YY(F\,\|\,\phi)]=(h')^2\,\E\,[\YY(F\,\|\,\wt\phi)\circ h] + \mu S_h,$$
respectively, where $S_h$ is Schwarzian derivative of $h,$ 
$$S_h = \Big( \frac{h''}{h'}\Big)' -\frac12\Big( \frac{h''}{h'}\Big)^2.$$

\subsubsec{Virasoro field} We define the \emph{Virasoro field} $T$ by 
$$T = -\frac12 J*J.$$
By Wick's calculus, it follows from \eqref{eq: EJJ} that
$$J(\zeta)J(z)=-\frac{1}{(\zeta-z)^2} + J\odot J(z) - \frac16 S_w(z) - \frac1{24} w'(z)^2+ o(1), \qquad(\zeta\to z),$$
where $w$ is a conformal map from $(D,q_-,q_+)$ onto $(\S,-\infty,\infty).$ 
Thus $T$ is expressed as
\begin{equation} \label{eq: TinS}
T = -\frac12 J\odot J + \frac1{12} S_w + \frac1{48} w'^2,
\end{equation}
and therefore $T$ is a Schwarzian form of order $1/12.$

\subsection{Lie derivatives and Ward family}
\subsubsec{Lie derivatives}
Recall that the Lie derivative $\LL_vX$ of $X$ against a vector field $v$ (see \cite[Section~4.4]{KM11}) is defined by 
$$(\LL_v X\,\|\, \phi) = \frac d{dt}\Big|_{t=0} (X\,\|\, \phi\circ\psi_{-t}),$$
where $\psi_t$ is a local flow of $v,$ and $\phi$ is an arbitrary chart. 
Recall the formula for the Lie derivatives of differentials and forms: 
\renewcommand{\theenumi}{\alph{enumi}}
{\setlength{\leftmargini}{1.7em}
\begin{enumerate}
\item $\LL_vX=\left(v\pa+\bar v\bar\pa+\lambda v'+\lambda_*\overline{v'}\right)X$ for a $(\lambda,\lambda_*)$-differential $X;$
\item $\LL_vX=\left(v\pa+v'\right)X +\mu v'$ for a pre-Schwarzian form $X$ of order $\mu;$ 
\item $\LL_vX=\left(v\pa+2v'\right)X +\mu v'''$ for a Schwarzian form $X$ of order $\mu.$
\end{enumerate}}
\noindent Denote
$$\LL_v^+ = \frac{\LL_v -i\LL_{iv}}2,\ \LL_v^- = \frac{\LL_v +i\LL_{iv}}2$$
so that $\LL_v = \LL_v^+ +\LL_v^-.$

\subsubsec{Ward family}
We denote by $\FF(T)$ the linear space of all Fock space fields $X$ such that for each non-random local vector field $v,$
$$\LL_vX= T^+_vX+ T^-_vX$$
holds in the maximal open set $D_\hol (v)$ where $v$ is holomorphic.
We call $\FF(T)$ the Ward family of $T.$
Here, $T^\pm_v$ are residue operators defined by
$$
(T_v^+X)(z)= \frac1{2\pi i}\oint_{(z)} vT\,X(z), \qquad
(T_v^-X)(z)= -\frac1{2\pi i}\oint_{(z)} \bar v\bar T\,X(z)
$$
in a given chart $\phi$ such that $\phi(p)=z.$ 
Recall that $X\in\FF(T)$ if and only if
Ward's OPE holds for $X$ in every local chart $\phi.$ That is to say, 
$$\Sing_{\zeta\to z}(T(\zeta)X(z))= (\LL_{k_\zeta}^+X)(z), \quad \Sing_{\zeta\to z}(T(\zeta)\bar X(z))= (\LL_{k_\zeta}^+\bar X)(z),$$
where $\Sing_{\zeta\to z}(T(\zeta)X(z))$ is the singular part of the operator product expansion $T(\zeta)X(z)$ as $\zeta\to z$ and $k_\zeta$ is the (local) vector field defined by $(k_\zeta\,\|\,\phi)(\eta) = 1/(\zeta-\eta)$ for a given chart $\phi:U\to\phi U.$
See \cite[Proposition~5.3]{KM11}.

For example, $\Phi$ and the fields in its OPE family $\FF$ belong to $\FF(T).$
(See the next proposition and \cite[Proposition~5.8]{KM11}.)
In particular, $T$ itself is in $\FF(T).$

\begin{prop} The Gaussian free field $\Phi$ is in the OPE family $\FF$ of $\Phi.$
\end{prop}
\begin{proof}
Denote $A= -\frac12 J\odot J.$
By Wick's calculus and \eqref{eq: EJPhi},
$$A(\zeta)\Phi(z) = -\E[J(\zeta)\Phi(z)]J(\zeta) \sim \frac{\pa\Phi(z)}{\zeta-z}$$
as $\zeta\to z.$
(We use the notation $\sim$ for the singular part of operator product expansion.)
Proposition now follows from \eqref{eq: TinS} and the fact that 
$$\LL_v^+X = (v\pa + \lambda v')X$$
for a $(\lambda,\lambda_*)$-differential $X.$
\end{proof}

\medskip \subsection{Ward's functionals and Ward's equations} \label{ss: Wv}
In this subsection we use Ward's functionals for dipolar Loewner vector fields to derive Ward's equations for the tensor product of fields in the OPE family $\FF$ of $\Phi.$
Combining Ward's equations with the level two degeneracy equation for the rooted vertex field, we will prove the main theorem in the next section.

It is convenient to use the $(\H,-1,1)$-uniformization.
For example,
\begin{equation} \label{eq: TinH}
T =A+ \frac1{12} S_w + \frac1{4} \frac{w'^2}{(1-w^2)^2},\qquad A:= -\frac12 J\odot J,
\end{equation}
where $w$ is a conformal map from $(D,q_-,q_+)$ onto $(\H,-1,1).$
It is easy to see that a quadratic differential $A$ has a simple pole at $q_\pm$ and the Virasoro field $T$ has a double pole at $q_\pm.$
Indeed, by \eqref{eq: EJPhi}, 
$$
\E\,[A(\zeta)\Phi(z_1)\odot \Phi(z_2)]=
\frac1\zeta \Big(-\frac{\sqrt{z_1^{\phantom{*}}}}{\zeta-z_1} + \frac{\sqrt{\bar z_1^{\phantom{*}}}}{\zeta-\bar z_1}\Big)\Big(-\frac{\sqrt{z_2^{\phantom{*}}}}{\zeta-z_2} + \frac{\sqrt{\bar z_2^{\phantom{*}}}}{\zeta-\bar z_2}\Big)
$$
in $\H$ with $q_-=0$ and $q_+=\infty.$

\subsubsec{Ward's functionals} 
We consider the following dipolar Loewner vector field in $D=\H:$
$$v_\zeta(z) = \frac{1-z^2}2\frac{1-\zeta z}{\zeta-z}, \qquad(\zeta\in\overline\H).$$
This meromorphic vector field has a simple zero at $q_\pm=\pm1.$
For $\zeta\in\H,$ Ward's functional $W^+(v_\zeta)$ is defined by
$$W^+(v_\zeta) = \lim_{\ve\to0} W^+(v_\zeta;U_\ve),$$
where $U_\ve = D\sm (B(q_-,\ve) \cup B(q_+,\ve)\cup B(\zeta,\ve))$ and
$$W^+(v_\zeta;U_\ve)=\frac1{2\pi i}\int_{\pa U_\ve} v_\zeta A-\frac1{\pi }\int_{U_\ve} (\bp v_\zeta)A,\qquad A:= -\frac12 J\odot J.$$
Then we have
\begin{equation} \label{eq: Wplus}
W^+(v_\zeta)=\frac1{2\pi i}\int_{\pa D} v_\zeta A-\frac1{\pi }\int_{D} (\bp v_\zeta)A,\qquad A:= -\frac12 J\odot J,
\end{equation}
where $\bp v$ is interpreted as a distribution.
As in the chordal case, for fields $X_j\in\FF$ and $z_j\in D\sm\{\zeta\},$ Ward's identity
\begin{equation} \label{eq: Ward}
\E\,[W^+(v_\zeta)\, X_1(z_1)\cdots X_n(z_n)] = \E\,[\LL^+_{v_\zeta}\,(X_1(z_1)\cdots X_n(z_n))]
\end{equation}
holds.

\subsubsec{Ward's equations} 
The following Lemma is the version of Ward's equation that we will use in the proof of Theorem~\ref{main}.
To state it, let us recall the modes $L_n$ of the Virasoro field $T:$
$$L_n(z):=\frac1{2\pi i}\oint_{(z)}(\zeta-z)^{n+1} \,T(\zeta)~d\zeta.$$
The operators $(L_nX)(z) = L_n(z)X(z)$ (in any given chart) are simply OPE multiplications,
\begin{equation} \label{eq: LT}
L_nX = T*_{(-n-2)}X.
\end{equation}
Denote $W^-(v) = \overline{W^+(v)}.$

\begin{lem} \label{Ward4Y}
Let $Y, X_1,\cdots, X_n\in \FF$ and let $X$ be the tensor product of $X_j$'s.
Then
\begin{align*}
\E\,[&Y(z)\LL_{v_z}^+X] + \E\,[\LL_{v_{\bar z}}^-(Y(z)X)]\\
&=\frac{(1-z^2)^2}{2}\,\E[(L_{-2}Y)(z)X]
-\frac{3z(1-z^2)}2\,\E[(L_{-1}Y)(z)X] \\
& +\frac{3z^2-1}2\,\E[(L_0Y)(z)X]+ \frac{z}2\,\E[(L_1Y)(z)X] - \frac1{8}\,\E[Y(z)X],
\end{align*}
where all fields are evaluated in the identity chart of $\H.$
\end{lem}

\begin{proof}
For $\zeta\in\H,$ the vector field $v_\zeta$ is meromorphic in $\wh\C$ without poles on $\R\cup\{\infty\}$ and its
reflected vector field
$$v_\zeta^\#(z)= \overline{v_{\zeta}(\bar z)} = v_{\bar\zeta}(z)$$
is holomorphic in $\H.$
By the definition of Ward's functional,
\begin{align} \label{eq: Wvzeta}
W^+(v_\zeta)&=\frac 1{2\pi i}\int_{-\infty}^\infty v_\zeta A-\frac1\pi\int_\H A~\bp v_\zeta,\\
W^+(v_{\bar\zeta})&=\frac1{2\pi i}\int_{-\infty}^\infty v_{\bar\zeta}A \nonumber.
\end{align}
Recall that $J$ is purely imaginary on the Dirichlet boundary and real on the Neumann boundary.
Thus $A$ is real on the whole boundary.
It follows that
\begin{equation} \label{eq: Wvbar}
\frac 1{2\pi i}\int_{-\infty}^\infty v_\zeta A=- W^-(v_{\bar\zeta}).
\end{equation}
By \eqref{eq: Wvzeta} and \eqref{eq: Wvbar},
$$\frac1\pi\int_\H A~\bp v_\zeta=-W^+(v_\zeta)-W^-(v_{\bar\zeta}).$$
Since $\bar\partial v_\zeta=-\pi/2 \,(1-\zeta^2)^2\delta_\zeta,$ we represent $A$ as Ward's functionals in the following way:
\begin{equation} \label{eq: A=}
 A(\zeta)=\frac2{(1-\zeta^2)^2}\,\Big(W^+(v_\zeta)+ W^-(v_{\bar\zeta})\Big).
\end{equation}

Next we express $\LL_{v_\zeta}^+Y(z)$ in terms of the OPE coefficients $ A*_jY(z)$ and the singular part of operator product expansion
$$A(\zeta)Y(z)\sim \sum_{j\le-1}{(A*_jY)(z)}{(\zeta-z)^j}, \qquad (\zeta\to z).$$
(We use the notation $\sim$ for the singular part of operator product expansion.)
Indeed, we have the following expression for $\LL_{v_\zeta}^+Y(z):$
\begin{align} \label{eq: sing OPE}
\LL_{v_\zeta}^+Y(z) &= \frac{(1-\zeta^2)^2}{2} \Sing_{\zeta\to z} (A(\zeta)Y(z)) \\
&+\frac{2\zeta+z-\zeta^3-\zeta^2z-\zeta z^2}2 \, A*_{-1}Y(z) \nonumber\\ &+\frac{1-\zeta^2-2\zeta z}2 A*_{-2}Y(z) -\frac\zeta2 \, A*_{-3}Y(z)\nonumber,
\end{align}
where $\Sing_{\zeta\to z} (A(\zeta)Y(z)) = \sum_{j\le-1}{(A*_jY)(z)}{(\zeta-z)^j}.$
Equation \eqref{eq: sing OPE} can be shown by the identity
\begin{align*}
\LL_{v_\zeta}^+Y(z) &= \frac1{2\pi i}\oint_{(z)}v_\zeta(\eta)A(\eta)Y(z)\,d\eta \\
&= \sum_{j\le-1}{A*_jY(z)} \frac1{2\pi i}\oint_{(z)}(\eta-z)^j\,v_\zeta(\eta)\,d\eta,
\end{align*}
and the fact that the integral 
$$\frac1{2\pi i}\oint_{(z)}(\eta-z)^j\,v_\zeta(\eta)\,d\eta
$$
is evaluated as 
\begin{align*}
&\frac{(1-\zeta^2)^2}2(\zeta-z)^j+\frac{2\zeta+z-\zeta^3-\zeta^2z-\zeta z^2}2 &\,\textrm{if } j=-1;\\
&\frac{(1-\zeta^2)^2}2(\zeta-z)^j+\frac{1-\zeta^2-2\zeta z}2 &\,\textrm{if } j=-2;\\
&\frac{(1-\zeta^2)^2}2(\zeta-z)^j-\frac\zeta2 &\,\textrm{if } j=-3;\\
&\frac{(1-\zeta^2)^2}2(\zeta-z)^j &\,\textrm{if } j\le-4.
\end{align*}
Subtracting the singular part of OPE and using \eqref{eq: sing OPE}, we have 
\begin{align*}
\E\,[(A\ast Y)(z)\,X]&=\lim_{\zeta\to z}\Big(\E\,[A(\zeta)Y(z)\,X]-\frac2{(1-\zeta^2)^2}\E\,[(\LL^+_{v_\zeta}Y)(z)\,X]\Big)\\
&+\frac2{(1-z^2)^2}\Big(\frac{3z(1-z^2)}2\,\E[(A*_{-1}Y)(z)X] \\
& +\frac{1-3z^2}2\,\E[(A*_{-2}Y)(z)X]- \frac{z}2\,\E[(A*_{-3}Y)(z)X]\Big).
\end{align*}
Let us now compute 
$$\lim_{\zeta\to z}\Big(\E\,[A(\zeta)Y(z)\,X]-\frac2{(1-\zeta^2)^2}\E\,[(\LL^+_{v_\zeta}Y)(z)\,X]\Big) .$$
It follows from \eqref{eq: A=}, Ward's identity~\eqref{eq: Ward}, and Leibniz's rule for Lie derivatives that
\begin{align*}
\E\,[A(\zeta) &Y(z)\,X]=\frac2{(1-\zeta^2)^2}\big(\E\,[W^+(v_\zeta) \,Y(z)\,X]+\E\,[W^-(v_{\bar\zeta}) \,Y(z)\,X]\big)\\
&=\frac2{(1-\zeta^2)^2}\big(\E\, [\LL^+_{v_\zeta}(Y(z)X)]+ \E\,[\LL^-_{v_{\bar\zeta}} (Y(z)X)]\big)\\
&=\frac2{(1-\zeta^2)^2}\big(\E\,[ (\LL^+_{v_\zeta} Y)(z)X]+\E\,[Y(z) \LL^+_{v_\zeta}X]+ \E\,[\LL^-_{v_{\bar\zeta}} (Y(z)X)]\big).
\end{align*}
Thus we have
\begin{align*}
\E\,&[Y(z)\LL_{v_z}^+X] + \E\,[\LL_{v_{\bar z}}^-(Y(z)X)]\\
&=\frac{(1-z^2)^2}{2}\,\E[(A*Y)(z)X]
-\frac{3z(1-z^2)}2\,\E[(A*_{-1}Y)(z)X] \\
& +\frac{3z^2-1}2\,\E[(A*_{-2}Y)(z)X]+ \frac{z}2\,\E[(A*_{-3}Y)(z)X].
\end{align*}
Lemma now follows from \eqref{eq: TinH} and \eqref{eq: LT}.
\end{proof}

\medskip \subsection{Bi-vertex fields and rooted vertex fields} \label{Vplus}
The Fock space field $e^{\odot ia\Phiplus(p,q_-)}$ in the main theorem (Theorem~\ref{main}) acts on the Fock space functionals/fields as a boundary condition changing operator.
This field does not belong to the Ward family $\FF(T).$
However, it can be represented as $V(p,q_-)/\E\,V(p,q_-)$ for some bi-vertex field $V$ (rooted at $q_-$) in $\FF(T).$

\subsubsec{Chiral bosonic field} 
We define a \emph{bi-variant} field $\Phiplus$ by
$$\Phiplus(z,z_0)=\big\{\Phiplus(\gamma)=\int_\gamma J(\zeta)\,d\zeta\big\},$$
where $\gamma$ is a curve from $z_0$ to $z.$ 
Then the values of $\Phiplus$ are multivalued functionals.
Integrating the operator product expansion 
$$T(\zeta)J(\eta)\sim\frac{J(\zeta)}{(\zeta-\eta)^2},$$
with respect to $\eta,$
we obtain Ward's OPE for $\Phiplus$ in both variables:
$$T(\zeta)\Phiplus(z, z_0)\sim \frac{J(z)}{\zeta-z}, \quad (\zeta\to z),\qquad T(\zeta)\Phiplus(z, z_0)\sim -\frac{J(z_0)}{\zeta-z_0}, \quad (\zeta\to z_0).$$
Thus it is natural to add $\Phiplus$ to $\FF.$

 By the definition of $\Phiplus,$
\begin{equation} \label{eq: G^+}
\E[\Phiplus(z, z_0)\Phi(z_1)]=2(\Gchiral(z,z_1)-\Gchiral(z_0,z_1)),
\end{equation}
where $\Gchiral$ is the \emph{complex} Green's function,
$$2\Gchiral(z,z_1)=G(z,z_1)+i\wt G(z,z_1).$$
(A multivalued function $\wt G$ is the harmonic conjugate of the Green's function.)
In terms of a conformal map $w:(D,q_-,q_+)\to (\S,-\infty,\infty),$ we have
$$\Gchiral(z,z_1)=\frac12\log \Big(\tanh\frac{w(z)-\overline{w(z_1)}}4\Big)-\frac12\log \Big(\tanh\frac{w(z)- w(z_1)}4\Big).$$
Differentiating \eqref{eq: G^+} with respect to $z_1$ we obtain
$$\E[\Phiplus(z, z_0)J(z_1)]=\frac12\frac{w'(z_1)}{\sinh((w(z)-w(z_1))/2)}-\frac12\frac{w'(z_1)}{\sinh((w(z_0)-w(z_1))/2)}.$$
It is a single-valued function of all three variables.
Correlations of two or more multivalued fields are defined only for \emph{non-intersecting paths}.
For instance, for non-intersecting paths $\gamma$ connecting from $z$ to $z_0$ and $\gamma'$ connecting from $z'$ to $z_0',$
$$\E[\Phiplus(\gamma)\Phiplus(\gamma')]=\log\frac{\tanh((w(z)-w(z'_0))/4)\tanh((w(z_0)-w(z'))/4)}{\tanh((w(z)-w(z'))/4)\tanh((w(z_0)-w(z'_0))/4)},$$
where the logarithm depends on the number of times $\gamma$ winds around $\gamma'.$

\subsubsec{Chiral bi-vertex fields} 
Similarly in the chordal case (see \cite[Section~12.2]{KM11}), we define the chiral vertex fields $V^\alpha(z,z_0), (z\ne z_0, \alpha\in\C)$ as bi-variant fields:
$$V^\alpha(z,z_0) = \Big(\frac{w'(z)w'(z_0)}{\tanh^2\dfrac{w(z)-w(z_0)}4}\Big)^{-\alpha^2/2}e^{\odot\alpha\Phiplus(z,z_0)},$$
where $w:(D,q_-,q_+)\to(\S,-\infty,\infty)$ is a conformal map. 
They are holomorphic differentials of conformal dimension 
$$\lambda=\lambda_0 = -\frac{\alpha^2}2$$
with respect to both variables $z$ and $z_0;$ they are $\Aut(D,q_-,q_+)$-invariant. 
They can be interpreted as the OPE exponential of bi-variant chiral bosonic field $\Phiplus(z,z_0)$ even though there is a technical difficulty in defining operator product expansions of multivalued fields.
Thus one can expect that Ward's OPE holds for $V^\alpha.$

\begin{prop}
We have
\begin{equation} \label{eq: OPE4V}
T(\zeta)V^\alpha(z,z_0) \sim -\frac{\alpha^2}2 \frac{V^\alpha(z,z_0)}{(\zeta-z)^2} + \frac{\pa_zV^\alpha(z,z_0)}{\zeta-z},\qquad (\zeta\to z).
\end{equation}
Similar operator product expansion holds as $\zeta\to z_0.$
\end{prop}

\begin{proof}
Since chiral vertex fields are differentials, it is sufficient to check Ward's OPE in the $(\S,-\infty,\infty)$-uniformization.
By Wick's calculus,
\begin{align*}
T(\zeta)V^\alpha(z,z_0)&= -\frac {\E\,[V^\alpha(z,z_0)]}2\sum_{n\ge0}\frac{\alpha^n}{n!}~(J(\zeta) \odot J(\zeta)) \, (\Phiplus(z,z_0))^{\odot n} \\
&\sim~-\alpha\, \E[J(\zeta)\Phiplus(z,z_0)] \,J(\zeta)\odot V^\alpha(z,z_0) \\
&-\frac12\alpha^2\, (\E[J(\zeta)\Phiplus(z,z_0)])^2\, V^\alpha(z,z_0).
\end{align*}
We need to compute the singular parts of correlation functions $\E[J(\zeta)\Phiplus(z,z_0)]$ and $(\E[J(\zeta)\Phiplus(z,z_0)])^2$ as $\zeta\to z.$
We have
$$\E[J(\zeta)\Phiplus(z,z_0)]=-\frac1{2\sinh\dfrac{\zeta-z}2}+\frac1{2\sinh\dfrac{\zeta-z_0}2} \sim -\frac1{\zeta-z}$$
and
$$(\E[J(\zeta)\Phiplus(z,z_0)])^2\sim \frac1{(\zeta-z)^2} -\frac1{\sinh\dfrac{z-z_0}2}\frac1{\zeta-z}$$
as $\zeta\to z.$
On the other hand, by direct computation,
$$\pa_zV^\alpha(z,z_0) = \alpha J(z)\odot V^\alpha(z,z_0) + V^\alpha(z,z_0) \frac{\pa_z\E\,[V^\alpha(z,z_0)]}{\E\,[V^\alpha(z,z_0)]},$$
where
$$\frac{\pa_z\E\,[V^\alpha(z,z_0)]}{\E\,[V^\alpha(z,z_0)]} = \frac{\alpha^2}2\frac1{\sinh\dfrac{z-z_0}2}.$$
This completes the proof.
\end{proof}

\subsubsec{Rooted vertex fields} 
We now freeze $z_0$ and consider $V^\alpha(z,z_0)$ as fields of one variable $z.$
With the choice $z_0 = q_-,$ we define the rooted vertex fields $V_\star^\alpha$ by
$$V_\star^\alpha(z) = V^\alpha(z,q_-)$$
provided that we require $w(\zeta)\sim-1/(\zeta-q_-)$ as $\zeta\to q_-$ in a fixed boundary chart at $q_-.$
In terms of a uniformization $w:(D,q_-,q_+)\to(\S,-\infty,\infty)$ with the above requirement, we have a simple expression
\begin{equation} \label{eq: Vstar}
V_\star^\alpha(z) = (w'(z))^{-\alpha^2/2}\,e^{\odot\alpha\Phiplus(z,q_-)}.
\end{equation}

 We now state Ward's identities for rooted vertex fields.
\begin{prop}
In the $(\H,-1,1)$-uniformization, we have
$$\E\,[W^+(v_\zeta)\, V_\star^\alpha(z) X_1(z_1)\cdots X_n(z_n)] = \E\,[\LL^+_{v_\zeta}(V_\star^\alpha(z) \,X_1(z_1)\cdots X_n(z_n))]$$
hold true provided that $X_j\in\FF$ and $z,z_j$'s are in $\H\sm\{\zeta\}.$
\end{prop}

\begin{proof}
We perform the computation in the $(\H,0,\infty)$-uniformization.
The definition of rooted vertex fields can be arrived to by the renormalization procedure:
$$V_\star^\alpha(z)=\lim_{z_0\to0}z_0^{-\alpha^2/2}V^\alpha(z,z_0).$$
We need to show that for a vector field $v$ with $v(0)=0,$
$$z_0^{-\alpha^2/2}(v(z_0)\pa_{z_0}V^\alpha(z,z_0)-\frac{\alpha^2}2v'(z_0)V^\alpha(z,z_0))$$
tends to zero (within correlations) as $z_0\to0.$
A direct computation shows that it converges to
$$V_\star^\alpha(z)\Big(\lim_{z_0\to0} \frac{\pa_{z_0}\E\,[V^\alpha(z,z_0)]}{\E\,[V^\alpha(z,z_0)]} v(z_0)-\frac{\alpha^2}2v'(z_0)\Big).$$
In the $(\H,0,\infty)$-uniformization,
$$\E\,[V^\alpha(z,z_0)]= \Big(\frac1{\sqrt{\mathstrut z}\sqrt{\mathstrut z_0}}\frac{\sqrt{\mathstrut z} + \sqrt{\mathstrut z_0}}{\sqrt{\mathstrut z} - \sqrt{\mathstrut z_0}}\Big)^{-\alpha^2}$$
and
$$\frac{\pa_{z_0}\E\,[V^\alpha(z,z_0)]}{\E\,[V^\alpha(z,z_0)]} = \frac{\alpha^2}2\frac1{z_0} +O(1)$$
as $z_0\to0.$
Proposition now follows since the vector field $v_\zeta$ has a simple zero at $q_-.$
\end{proof}

\subsubsec{Ward's equations and level two degeneracy equation for $V_\star^{ia}$} 
Combining the argument in the proof of Lemma~\ref{Ward4Y} and the previous proposition,
we have the following version of Ward's equations.
(Since the rooted vertex fields are holomorphic, $\LL_{v_{\bar z}}^-V_\star^\alpha(z)=0.$)

\begin{lem} \label{Ward}
Let $X_1,\cdots, X_n\in \FF$ and let $X$ be the tensor product of $X_j$'s.
Then
\begin{align*}
\E\,[&V_\star^\alpha(z)(\LL_{v_z}^+X +\LL_{v_{\bar z}}^-X)]\\
&=\frac{(1-z^2)^2}{2}\,\E[(L_{-2}V_\star^\alpha)(z)X]
-\frac{3z(1-z^2)}2\,\E[(L_{-1}V_\star^\alpha)(z)X] \\
& +\frac{3z^2-1}2\,\E[(L_0V_\star^\alpha)(z)X]+ \frac{z}2\,\E[(L_1V_\star^\alpha)(z)X] - \frac1{8}\,\E[V_\star^\alpha(z)X],
\end{align*}
where all fields are evaluated in the identity chart of $\H.$
\end{lem}

 We now compute $L_nV_\star^\alpha\,(n=-2,-1,0,1)$ in the case when $\alpha = ia\,(a = 1/\sqrt{2}).$
From the algebraic point of view, level two degeneracy equation for $V_\star^{ia}\,(a = 1/\sqrt{2})$ follows from the fact that $V_\star^{ia}$ is a holomorphic Virasoro primary field with conformal dimension $a^2/2$ and a current primary field with charge $a$ (see \cite[Appendix~11]{KM11}).
We rather present the direct and analytic proof of level two degeneracy equation for $V_\star^{ia}.$

\begin{lem}\label{level2} If $a=1/\sqrt2,$ then
$$T*V_\star^{ia}=\partial^2 V_\star^{ia}.$$
\end{lem}

\begin{proof}
We write $V$ for $V_\star^{ia}.$
Recall that 
$V(z) = (w'(z))^{a^2/2}\,e^{\odot ia\Phiplus(z,q)}, q = q_-,$ and 
$$T = A+ \frac{1}{12}S_w+ \frac1{48}w'^2,\quad A = -\frac12 J\odot J,$$
where $w$ is a conformal map from $(D,q_-,q_+)$ onto $(\S,-\infty,\infty).$ 
It follows from direct computation that 
$$\pa V = iaJ \odot V+\frac{a^2}2\frac{w''}{w'} V,$$
and
$$\pa^2 V = A \odot V + ia(\pa J )\odot V + \frac{ia}{2}\frac{w''}{w'} J \odot V
+\frac{a^2}2\Big(\frac{w'''}{w'} -\frac34\frac{w''^2}{w'^2}\Big)V.$$
On the other hand, by Wick's calculus,
\begin{align*}
A(\zeta)V(z)&= A(\zeta)\odot V(z)-ia\,\E[J(\zeta)\Phiplus(z,q)] \,J(\zeta)\odot V(z) \\
&+\frac12 a^2\, \E[J(\zeta)\Phiplus(z,q)]^2\, V(z).
\end{align*}
As $\zeta\to z,$ up to $o(1)$ terms, we have
$$\E[J (\zeta)\Phiplus (z,q)]=-\frac{w'(\zeta)}{2\sinh\dfrac{w(\zeta)-w(z)}2} = -\frac1{\zeta-z} - \frac12\frac{w''(z)}{w'(z)}$$
and
$$(\E[J (\zeta)\Phiplus (z,q)])^2=\frac1{(\zeta-z)^2} + \frac{w''(z)}{w'(z)}\frac1{\zeta-z} -\frac1{12}w'(z)^2 - \frac14\frac{w''(z)^2}{w'(z)^2}+\frac23\frac{w'''(z)}{w'(z)}.$$
Thus we have
$$A*V = A \odot V+ia(\pa J )\odot V+\dfrac{ia}2\dfrac{w''}{w'} J \odot V
+ \frac{a^2}2 \Big(-\frac1{12}w'^2 - \frac14\frac{w''^2}{w'^2}+\frac23\frac{w'''}{w'}\Big)V.
$$
Lemma now follows since $S_w = (w''/w')'- (w''/w')^2/2 = w'''/w'- 3(w''/w')^2/2.$
\end{proof}

\begin{rmk*}
It is immediate from the proof of the previous lemma that 
$$L_1V = 0, \qquad L_0V = h V, \qquad L_{-1}V = \pa V.$$
For example,
$$L_{-1}V = T*_{-1}V = iaJ \odot V+\frac{a^2}2\frac{w''}{w'} V = \pa V.$$
\end{rmk*}

\medskip \section{Connection to dipolar SLE(4)}
After we briefly study how the insertion of Wick's exponential $e^{\odot ia\Phiplus(p,q_-)}$ acts as a boundary condition changing operator on Fock space fields, we prove the main theorem that the correlation functions of fields in the OPE family $\FF$ of the Gaussian free field $\Phi$ with mixed boundary condition under the insertion of $e^{\odot ia\,\Phiplus(p,q_-)}\,(a = 1/\sqrt2)$ are martingale-observables for dipolar SLE(4).

\medskip \subsection{Boundary condition changing operator}

The insertion of $V_\star^{ia}(p)\,(a = 1/\sqrt2)$ produces an operator
$$\XX\mapsto\widehat\XX$$
on Fock space functionals/fields.
By definition, this correspondence is given by the formula
\begin{equation} \label{eq: BC} \index{bosonic field $\Phi, \widehat\Phi$}
\widehat\Phi=\Phi+2ia\,\Gchiral(p,z), \qquad (a = \frac1{\sqrt2}),
\end{equation}
and the rules
\begin{equation} \label{eq: BCrules}
\pa\XX\mapsto\pa\widehat\XX,\quad \bp\XX\mapsto\bp\widehat\XX,\quad\alpha\XX+\beta\YY\mapsto \alpha\widehat\XX+\beta\widehat\YY,\quad \XX\odot\YY\mapsto\widehat\XX\odot\widehat\YY.
\end{equation}
If $w:(D, p,q_-,q_+)\to (\S, 0,-\infty,\infty)$ is a conformal map, then
$$2ia\,\Gchiral(p,z)=2a\arg \tanh \frac{w(z)}4.$$
We denote by $\widehat \FF$ the image of $\FF$ under this correspondence and denote
$$\widehat \E[\XX]:=\frac{\E [V_\star^{ia}(p)\XX]}{\E [V_\star^{ia}(p)]}=\E [e^{\odot ia \Phiplus (p,q_-)}\XX].$$

As in \cite{KM11}, one can prove the following proposition by induction.
\begin{prop}
Let $\widehat\XX\in \widehat \FF$ correspond to the string $\XX\in \FF$ under the map given by \eqref{eq: BC} and \eqref{eq: BCrules}.
Then
\begin{equation} \label{eq: hat E=E hat}
\widehat \E[\XX]=\E[\widehat\XX].
\end{equation}
\end{prop}

\medskip \noindent\textbf{Examples:} \quad
\renewcommand{\theenumi}{\alph{enumi}}
{\setlength{\leftmargini}{2.0em}
\begin{enumerate}
 \item The current $\widehat J$ is a $(1,0)$-differential,
$$\widehat J = J -\frac{ia}2\frac{w'}{\sinh(w/2)}.$$
In the $(\S,0,-\infty,\infty)$-uniformization, 
$$\hat\jmath(z):=\widehat \E\, J(z)=-\dfrac{ia}2\dfrac1{\sinh(z/2)};$$
\item The Virasoro field $\widehat T$ is a Schwarzian form of order $1/{12},$
\begin{align*}
\widehat T &= -\dfrac12\widehat J*\widehat J = T + \frac{ia}2\frac{w'}{\sinh{(w/2)}} J+ \frac{a^2}8\Big(\frac{w'}{\sinh{(w/2)}} \Big)^2 \\
&= -\frac12 J\odot J + \frac{ia}2\frac{w'}{\sinh(w/2)} J+ \dfrac{1}{12}Sw + \Big(\frac1{48} + \frac1{16}\frac{1}{\sinh^2{(w/2)}} \Big) w'^2.
\end{align*}
In the $(\S,0,-\infty,\infty)$-uniformization, 
$$\widehat \E\, T(z)= \dfrac1{48} + \dfrac1{16}\dfrac{1}{\sinh^2{(z/2)}} ;$$
\item The non-chiral vertex field $\widehat\VV^\alpha$ is a $(-\alpha^2/2,-\alpha^2/2)$-differential,
$$\widehat\VV^\alpha=e^{2\alpha a\arg \tanh\frac w4}\VV^\alpha = e^{2\alpha a\arg \tanh\frac w4}C^{\alpha^2}e^{\odot\alpha\Phi}.$$
In the $(\S,0,-\infty,\infty)$-uniformization,
$$\widehat \E\, \VV^\alpha(z)=(4\tan\frac y2)^{\alpha^2}e^{2\alpha a\arg \tanh\frac z4};$$
\item The bi-vertex field $\widehat V^\alpha(z, z_0)$ is a $-\alpha^2/2$-differential in both variables,
$$\qquad\widehat V^\alpha(z, z_0)=\left(\dfrac{w'w'_0}{\tanh^2((w-w_0)/4)}\right)^{-\alpha^2/2}\,\left(\dfrac{\tanh(w/4)}{\tanh(w_0/4)}\right)^{-i\alpha a}\,e^{\odot \alpha\Phiplus(z,z_0)},$$
where $w = w(z), w_0=w(z_0),w' = w'(z), w_0'=w'(z_0).$
We have 
$$\widehat \E \,V^\alpha(z, z_0)=(\tanh((z-z_0)/4))^{\alpha^2}(\tanh(z/4)/\tanh(z_0/4))^{-i\alpha a}$$
in the $(\S,0,-\infty,\infty)$-uniformization.
\end{enumerate}}

\medskip \subsection{Dipolar SLE(4) martingale-observables}
We now prove the main theorem.
\begin{proof}[Proof of Theorem~\ref{main}]
Let us denote
$$R_\xi(z_1,\cdots, z_n)\equiv\wh\E_\xi[X]=\E[e^{\odot ia \Phiplus(\xi,q_-)}X],$$
where $X = X_1(z_1)\cdots X_n(z_n).$
We represent the process
$$M_t(z_1,\cdots, z_n)=M_{D_t,\gamma_t,q_-,q_+}(z_1,\cdots, z_n)$$
as 
$$M_t = m(\wt\xi_t,t), \quad m(\xi,t) =\big(R_\xi\,\|\,\wt g_t^{-1}\big), \quad \wt\xi_t = \frac{e^{2 B_t}-1}{e^{2 B_t}+1},\quad \wt g_t = \frac{e^{g_t}-1}{e^{g_t}+1}.$$
Applying It\^o's formula (note that the function $m(\xi,t)$ is smooth in both variables) to the process $m(\wt\xi_t,t)$, it is easy to see that the drift term of $dM_t$ is equal to 
$$\Big(\big(\frac{(1-\xi^2)^2}2\pa_\xi^2-\xi(1-\xi^2)\pa_\xi\big)\Big|_{\xi=\wt\xi_t}~m(\xi,t)\Big)\,dt+L_t\, dt,$$
where
$$
L_t:=\frac d{ds}\Big|_{s=0}\big(R_{\wt\xi_t}\,\|\,\wt g_{t+s}^{-1}\big)
=\frac d{ds}\Big|_{s=0}\big(R_{\wt\xi_t}\,\|\,\wt g_t^{-1}\circ f_{s,t}^{-1}\big),\qquad (f_{s,t} = \wt g_{t+s}\circ \wt g_t^{-1}).
$$
Since the time-dependent flow $f_{s,t}$ satisfies
$$\frac{d}{ds}f_{s,t}(\zeta) = -\frac{1-f_{s,t}^2(\zeta)}2\frac{1-\wt\xi_{t+s}f_{s,t}(\zeta)}{\wt\xi_{t+s}-f_{s,t}(\zeta)},$$
we have an approximation of $f_{s,t}:$
$$f_{s,t} = \id - sv_{\wt\xi_t} + o(s)\qquad(\textrm{as } s\to0),\qquad v_{\xi}(z)=\frac{1-z^2}2\frac{1-\xi z}{\xi-z}.$$

We can represent $L_t$ in terms of Lie derivative of $R_{\wt\xi_t}$ in the global chart $\wt g_t^{-1}$
(since the fields in $\FF$ depend smoothly on local charts):
$$L_t=-\big(\LL_{v_{\wt\xi_t}}R_{\wt\xi_t}\,\|\,\wt g_{t}^{-1}\big).$$
The driftlessness of $dM_t$ follows from the BPZ-Cardy equation:
\begin{equation} \label{eq: BPZ-Cardy}
\wh\E_\xi[\LL_{v_\xi}X]=\Big(\frac{(1-\xi^2)^2}2\pa_\xi^2-\xi(1-\xi^2)\pa_\xi\Big) \wh\E_\xi [X],
\end{equation}
where $\pa_\xi $ is the operator of differentiation with respect to the real variable $\xi.$

To show \eqref{eq: BPZ-Cardy}, denote
$$V = V_\star^{ia},\qquad (a = \frac1{\sqrt{2}}).$$
In terms of a uniformizing map $w:(D,p,q_-,q_+) \to (\H,0,-1,1),$ (up to constant)
$$V(z) = \Big(\frac{w'(z)}{1-w(z)^2}\Big)^{a^2/2}\,e^{\odot ia\Phiplus(z,q_-)}$$
(cf. \eqref{eq: Vstar}).
Since $R_\xi$ does not depend on the boundary chart at $q_-=-1$ where we normalize $V_\star^{ia},$ we can assume that for $\xi\in(-1,1)$
$$V(\xi) = (1-\xi^2)^{-h} e^{\odot ia\Phiplus[\tau]}, \qquad(h=\frac{a^2}2 = \frac14),$$
where $\tau$ is the half-line $(-1,\xi).$ 
Therefore, $R_\xi=\E[(1-\xi^2)^{h}V(\xi)X].$
For $\zeta$ close to $\xi,$ denote
$$R_\zeta\equiv R(\zeta;z_1,\cdots, z_n)=\E [(1-\zeta^2)^h V(\zeta)X],\qquad (\zeta\in\H),$$
where we use a path from $q_-=-1$ to $\zeta$ so that $R_\xi = \lim_{\zeta\to \xi} R_\zeta.$

By Lemma~\ref{Ward} (Ward's equation for $V$), Lemma~\ref{level2} (the level two degeneracy equations for $V$), and its Remark ($L_1V = 0, L_0V = h V, L_{-1}V = \pa V$), we have 
\begin{align*}
(1-\zeta^2)^h (\E\,[&V(\zeta)\LL_{v_\zeta}^+X] + \E\,[V(\zeta)\LL_{v_{\bar \zeta}}^-X])=\Big(\frac{(1-\zeta^2)^2}2\pa_\zeta^2-\zeta(1-\zeta^2)\pa_\zeta\Big)R_\zeta.
\end{align*}
Since $\xi$ is real, the left-hand side converges to
$$\E\,[(1-\xi^2)^h V(\xi)\LL_{v_\xi}X] = \wh\E_\xi[\LL_{v_\xi}X]$$
as $\zeta\to \xi.$
On the other hand, since the field $V$ is holomorphic and $\pa_\xi = \pa +\bp,$ the right-hand side converges to
$$\Big(\frac{(1-\xi^2)^2}2\pa_\xi^2-\xi(1-\xi^2)\pa_\xi\Big)R_\xi$$
as $\zeta\to \xi.$
Thus we get \eqref{eq: BPZ-Cardy}.
\end{proof}

\begin{eg*}
The 1-point function $\wh\E\,\Phi(z)$ is Schramm's observable, i.e.,
$$\P[\,z~\textrm{is to the left of SLE(4) curve}~\gamma~\textrm{in}~\S\,] = \frac1\pi \arg\tanh\frac z4.$$
In particular, the above probability with $z=x+\pi i\in\R+\pi i$ gives the distribution of $\gamma(\infty).$ See \cite{BBH05} and \cite{Zhan04}.
\end{eg*}

\medskip 
\def\cprime{$'$}


\end{document}